\documentclass[12pt]{amsart}
\newtheorem{thm}{Theorem}[section]
\newtheorem{prop}[thm]{Proposition}

\newtheorem{cor}[thm]{Corollary}

\theoremstyle{definition}

\theoremstyle{remark}

\numberwithin{equation}{section}

\begin{document}

%%
%% The title of the paper goes here.  Edit to your title.
%%

\title{The Kth Traveling Salesman Problem is Pseudopolynomial when TSP is polynomial}

%%
%% Now edit the following to give your name and address:
%% 

\author{Brahim Chaourar}
\address{Department of Mathematics and Statistics, Al Imam \iffalse Mohammad Ibn Saud Islamic \fi University (IMSIU), P.O. Box
90950, Riyadh 11623,  Saudi Arabia \\ Correspondence address: P.O. Box 287574, Riyadh 11323, Saudi Arabia}
\email{bchaourar@hotmail.com}
%\urladdr{www.math.sc.edu/$\sim$howard} % Delete if not wanted.

%%
%% If there is another author uncomment and edit the following.
%%

%\author{Second Author}
%\address{Department of Mathematics, University of South Carolina,
%Columbia, SC 29208}
%\email{second@math.sc.edu}
%\urladdr{www.math.sc.edu/$\sim$second}

%%
%% If there are three of more authors they are added in the obvious
%% way. 
%%

%%%
%%% The following is for the abstract.  The abstract is optional and
%%% if not used just delete, or comment out, the following.
%%%

\begin{abstract}
Given an undirected graph $G=(V, E)$ with a weight function $c\in R^E$, and a positive integer $K$, the Kth Traveling Salesman Problem (KthTSP) is to find $K$ Hamilton cycles $H_1, H_2, , ..., H_K$ such that, for any Hamilton cycle $H\not \in \{H_1, H_2, , ..., H_K \}$, we have $c(H)\geq c(H_i), i=1, 2, ..., K$. This problem is NP-hard even for $K$ fixed. We prove that KthTSP is pseudopolynomial when TSP is polynomial.
\end{abstract}

%%
%%  LaTeX will not make the title for the paper unless told to do so.
%%  This is done by uncommenting the following.
%%

\maketitle

%%
%% LaTeX can automatically make a table of contents.  This is done by
%% uncommenting the following:
%%

%\tableofcontents

%%
%%  To enter text is easy.  Just type it.  A blank line starts a new
%%  paragraph. 

{\bf2010 Mathematics Subject Classification:} Primary 90C27, Secondary 90C57.
\newline {\bf Key words and phrases:} K best solutions, Traveling Salesman Problem, Kth best Traveling Salesman Problem, pseudopolynomial.

\iffalse

%%
%%  To put mathematics in a line it is put between dollor signs.  That
%%  is $(x+y)^2=x^2+2xy+y^2$
%%

Anyone caught using formulas such as $\sqrt{x+y}=\sqrt{x}+\sqrt{y}$ 
or $\frac{1}{x+y}=\frac{1}{x}+\frac{1}{y}$ will fail.

%%
%%% Displayed mathematics is put between double dollar signs.  
%%

The binomial theorem is
$$
(x+y)^n=\sum_{k=0}^n\binom{n}{k}x^ky^{n-k}.
$$
A favorite sum of most mathematicians is
$$
\sum_{n=1}^\infty \frac{1}{n^2}=\frac{\pi^2}{6}.
$$
Likewise a popular integral is
$$
\int_{-\infty}^\infty e^{-x^2}\,dx=\sqrt{\pi}
$$

%%
%% A Theorem is stated by
%%

\begin{thm} The square of any real number is non-negative.
\end{thm}

%%
%% Its proof is set off by
%% 

\begin{proof}
Any real number $x$ satisfies $x>0$, $x=0$, or $x<0$.
If $x=0$, then $x^2=0\ge 0$.  If $x>0$ then as a positive time a
positive is positive we have $x^2=xx>0$.  If $x<0$ then $-x>0$ and so
by what we have just done $x^2=(-x)^2>0$.  So in all cases $x^2\ge0$.
\end{proof}

%%
%% A new section is started as follows:
%%

\fi

%%%%%%%%%%%%%%%%%%%%%%%%%%%%%%%%%%%%%%%%%%%%%%%%%%%%%%%%%%%%%%%%%%%%%%
\section{Introduction}
%%%%%%%%%%%%%%%%%%%%%%%%%%%%%%%%%%%%%%%%%%%%%%%%%%%%%%%%%%%%%%%%%%%%%%

\textbf{Sets and their characterisitic vectors will not be distinguished.} We refer to Bondy and Murty \cite{Bondy and Murty 2008} and Schrijver \cite{Schrijver 1986} about, respectively, graph theory and polyhedra terminolgy and facts.
\newline Given an undirected graph $G=(V, E)$ with a weight function $c\in R^E$, and a positive integer $K$, the Kth Traveling Salesman Problem (KthTSP) is to find $K$ distinct Hamilton cycles $H_1, H_2, , ..., H_K$ such that, for any Hamilton cycle $H\not \in \{H_1, H_2, , ..., H_K \}$, we have $c(H)\geq c(H_i), i=1, 2, ..., K$. Since KthTSP is the famous TSP for $K=1$, then KthTSP is NP-hard even for $K$ fixed. KthTSP is motivated by searching near optimal solutions with some special properties: when in addition of the TSP comstraints, ''there are some other wich might be difficult to consider explicitly in a mathematical model, or if considered, would increase largely the size of the model. By finding the best, second best, ..., Kth best solution, we are able to sequentially verify these solutions with respect to the additional constraints and stop when a solution that satisfies all of them is found'' \cite{Yanasse 2000}. Another motivation is that if, for any reason, the route of the best solution is unavailable, then alternate solutions (routes) are desirable \cite{Pollack 1961}.
\newline Finding K best solutions of an optimization problem in general has been studied by few authors \cite{Hamacher and Queyranne 1985, Murty 1968a, Murty 1968b, Wolsey 1973} and almost the same situation happened for particular problems \cite{Camerini et al. 1975, Camerini et al. 1980a, Camerini et al. 1980b, Chaourar 2008, Chaourar 2010, Gabow 1977, Hamacher et al. 1984, Katoh et al. 1981, Megiddo et al. 1981, Murty 1968b}.
\newline The remainder of the paper is organized as follows: in section 2, we give an algorithm for finding K best solutions for a general model containing KthTSP, then, in section 3, we apply this algorithm to KthTSP and deduce that it is polynomial on $K$ and $|E|$ when TSP is polynomial. And we conclude in section 4.

%%%%%%%%%%%%%%%%%%%%%%%%%%%%%%%%%%%%%%%%%%%%%%%%%%%%%%%%%%%%%%%%%%%%%%
\section{An Algorithm for Finding K Best Solutions of a Large Class of Combinatorial Optimization Problems}
%%%%%%%%%%%%%%%%%%%%%%%%%%%%%%%%%%%%%%%%%%%%%%%%%%%%%%%%%%%%%%%%%%%%%%

Let $P\subseteq R^m$ be a polyhedra, $f(m)$ be the number of its facets, $N(x)$ be the set of all neighbors of an extreme point $x\in P$, $x_K$ the Kth best solution in $P$, regarding to a given weight function and a given positive integer $K$.
\newline Based on the following property, an algorithm has been used for particular problems \cite{Chaourar 2008, Chaourar 2010}.
\begin{prop}
For any positive integer $j$ such that $2\leq j\leq K$, 
$$x_j\in \bigcup_{i=1}^{j-1} N(x_i)\backslash \{x_1, x_2, ..., x_{j-1}\}.$$
\end{prop}
Since selecting $K$ best numbers from a list of $n$ numbers requires a running time complexity of $O(n+K log K)$ \cite{Cormen et al. 2009}, solving an $n\times n$ system of linear equations is $O(n^3)$ \cite{Farebrother 1988}, and if $C(m)$ is the running time complexity for finding the best solution on $P$, then we have the following two consequences.
\begin{cor}
The running time complexity for finding K best solutions of $P$ regarding to a given weight function is $O(C(m)+K N m^3+K log K)$ where $N$ is the maximum cardinality of all $N(x_i), i=1, ..., K-1$.
\end{cor}
Since $N$ can be bounded by $m f(m)-m^2$ then:
\begin{cor}
The running time complexity for finding K best solutions of $P$ regarding to a given weight function is $O(C(m)+K m^4 f(m)+K log K)$.
\end{cor}
\begin{cor}
If $C(m)$ and $f(m)$ are polynomial on $m$ then finding K best solutions of $P$ is pseudopolynomial, i.e., polynomial on $m$ and $K$.
\end{cor}
We will propose now a new algorithm which generalizes one used in \cite{Hamacher and Queyranne 1985} for the Kth Best Base of a Matroid (KBBM).
\newline Let us give a general model of combinatorial objects containing Hamilton cycles.
\newline Let $E$ be a finite set and $\mathcal X=\subseteq \{0, 1\}^E$.
We say that $\mathcal X$ is an $\alpha$-bases system, where $\alpha$ is a positive integer, if the following conditions hold:
\begin{enumerate}
\item $\alpha =Min\{x\backslash y$ such that $(x, y)\in \mathcal X^2, x\neq y\}$;
\item there exists a positive integer $r$ such that $x(E)=r$, for any $x\in \mathcal X$;
%\item for any $x\in \mathcal X$, and for any $F\subseteq x$, with $|F|\geq \alpha$, there exists $F'\subseteq x\backslash F$ and $F''\subseteq E\backslash x$ such that $|F'|\leq 1$ and $x\backslash (F\cup F')\cup F''\in \mathcal X$.
\item for any $(x, x')\in \mathcal X^2$, there exist $t\in N$, $F'_i\subseteq x'\backslash x$, and $F_i\subseteq x\backslash x', i=1, 2, ..., t$ such that $\alpha \leq |F_i|=|F'_i|\leq \alpha +1$, $x'=x\backslash (\bigcup_{i=1}^t F_i)\cup  (\bigcup_{i=1}^t F'_i)$ and $x\backslash (\bigcup_{i\in I\subseteq \{1, ..., t\}} F_i)\cup  (\bigcup_{i\in I\subseteq \{1, ..., t\}} F'_i)\in \mathcal X$.
\end{enumerate}
Such pair $(F_i, F'_i)$ verifying the condition (3) is called an $x$-exchangeable pair.
\newline Note that bases of a matroid form a 1-bases system and we will prove that Hamilton cycles of a complete graph form a 2-bases system.
\newline We have then the following property for K best solutions of $\alpha$-bases system.
\begin{thm}
Given a weight function $c\in R^E$ and a jth $c$-best solution (of $\mathcal X$) $x$. If $(F_0, F'_0)$ is an $x$-exchangeable pair such that $c(F_0)-c(F'_0)=Maximum\{c(F)-c(F')$ such that $(F, F')$ is an $x$-exchangeable pair and $c(F)-c(F')\leq 0$\}, then $x_0=(x\backslash F_0)\cup F'_0$ is a (j+1)th $c$-best solution of $\mathcal X$. 
\end{thm}
\begin{proof} By induction on $j\geq 1$.
\newline By using the condition (3) of the definition of $\alpha$-bases systems, any $x'\in \mathcal X\backslash \{x\}$ can be expressed as $x'=x\backslash (\bigcup_{i=1}^t F_i)\cup  (\bigcup_{i=1}^t F'_i)$ for some $x$-exchangeable pairs $(F_i, F'_i)\subseteq (x\backslash x')\times(x'\backslash x), i=1, ..., t$. Since $x$ is a $c$-best solution then  $\mathcal F_x=\{(F, F')$ $x$-exchangeable pairs such that $c(F)-c(F')>0$\}=$\O$. Thus $c(x_0)=c(x)-(c(F_0)-c(F'_0))\leq c(x)-\sum_{i=1}^t (c(F_i)-c(F'_i))=c(x')$. So $x_0$ is the 2nd $c$-best solution.
\newline Suppose now that $j\geq 2$ and let $x_i$ be the ith $c$-best solution for $i=1, 2, ..., j$.
\newline For any subset $X\subseteq \mathcal F_x$ we can get a $x_i=x\backslash (\bigcup_{(F, F')\in X} F)\cup  (\bigcup_{(F, F')\in X} F')$ and $c(x_i)=c(x)-\sum_{(F, F')\in X} (c(F)-c(F'))\leq c(x)$ ($X=\O$ gives $x=x_j$ itself and $X=\mathcal F_x$ gives the $c$-best solution). It follows that $x_{j+1}=x_0=x\backslash F_0\cup F'_0$ because of a similar argument as for $j=1$.
\end{proof}
This proof gives an algorithm for finding K best solutions in $\alpha$-bases systems. The algorithm consists of finding the best solution first ($O(C(m))$) and then the 2nd best by adding a subset to the (best) solution ($O(|E|-r)$), finding the matched subsets of our (best) solution forming an $x$-echangeable pair ($O(\theta)$) and choosing the best subset of this solution forming an exchangeable pair ($O(r)$). By repeating this procedure $K$ times, the running time complexity of this algorithm is $O(C(m)+K r (|E|-r) \theta)$ where $\theta$ is the running time complexity of the oracle used to find exchangeable pairs.

%%%%%%%%%%%%%%%%%%%%%%%%%%%%%%%%%%%%%%%%%%%%%%%%%%%%%%%%%%%%%%%%%%%%%%
\section{KthTSP is pseudopolynomial when TSP is polynomial}
%%%%%%%%%%%%%%%%%%%%%%%%%%%%%%%%%%%%%%%%%%%%%%%%%%%%%%%%%%%%%%%%%%%%%%

First we need to prove that Hamilton cycles of a complete graph verify the properties (1)-(3) of $\alpha$-bases systems.
\begin{thm}
Hamilton cycles of a complete graph form a 2-bases system.
\end{thm}
\begin{proof} For Hamilton cycles, E is the set of edges of a given complete graph $K_n$.
\newline \textbf{Property (1):} It is clear that $\alpha=2$.
\newline \textbf{Property (2):} It is clear that $r=n$.
\newline \textbf{Property (3):} Let $H$ and $H'$ two distinct Hamilton cycles and $d(H, H')=|H\backslash H'|$. We will prove this property by induction on $d(H, H')$.
\newline If $d(H, H')=2$ (respectively 3) then let $F=H\backslash H'$ and $F'=H'\backslash H$. It is not difficult to see that $H'=(H\backslash F)\cup F'$, $|F|=|F'|=2=\alpha$ (respectively $=3=\alpha +1$) and $(F, F')$ is an $H$-exchangeable pair.
\newline If $d(H, H')=2p$ (respectively $2p+1$), with $p\geq 2$, then there exists a circuit $C=\{e, e', f, f'\}$ (of cardinality 4) such that $\{e, f\}\subseteq H\backslash H'$, $e'\in H'\backslash H$, $f'\notin H$ and $H''=H\Delta C=H\backslash \{e, f\}\cup \{e', f'\}$ is a Hamilton cycle. It is clear that $d(H'', H')\leq d(H, H')-1$. By induction, $H'$ can be expressed in means of $H''$ and $H''$-exchangeable pairs. If $f'\in H'$ then we are done. Else, one of the removed $H''$-exchangeable pairs should contain $f'$ and by subsituting $H''$, we will get an $H$-exchangeable pair with components of cardinality 3.
\end{proof}
Since finding exchangeable pairs corresponds to choose 2 nonadjacent edges (respectively 3 edges) from a Hamilton cycle and to find 2 nonadjacent edges (respectively 3 edges) such that exchanging between them gives a new Hamilton cycle then $O(\theta)=O(1)$. It follows that the running time complexity of our algorithm for KthTSP is $O(C(m)+K n m)$. Then we can state our main result.
\begin{cor}
KthTSP is pseudopolynomial when TSP is polynomial.
\end{cor}
\begin{proof}
If TSP is polynomial for (special instances of) complete graphs then $C(m)$ is polynomial and we are done.
\newline If TSP is polynomial for special classes of graphs, then we can put an infinity weight to removed edges from the corresponding complete graph and we get the same result.
\end{proof}
Note that, with a natural modification, our algorithm works for arbitrary weights and for Max KthTSP.

%%%%%%%%%%%%%%%%%%%%%%%%%%%%%%%%%%%%%%%%%%%%%%%%%%%%%%%%%%%%%%%%%%%%%%
\section{Conclusion}
%%%%%%%%%%%%%%%%%%%%%%%%%%%%%%%%%%%%%%%%%%%%%%%%%%%%%%%%%%%%%%%%%%%%%%

We have generalized an algorithm described in \cite{Hamacher and Queyranne 1985} for a generalization of bases of a matroid. By applying this algorithm to Hamilton cycles, we have proved that KthTSP is psudopolynomial when TSP is polynomial. Future investigations can be applying this algorithm for appropriate combinatorial objects.

\end{document}